\documentclass[11pt,a4paper]{amsart} 
\usepackage[all]{xy}
\usepackage{amssymb}
\SelectTips{cm}{}
\calclayout \makeatletter
\def\serieslogo@{} 
\def\@setcopyright{} 
\makeatother

\title[The telescope conjecture for hereditary rings]{The telescope
conjecture for hereditary rings via Ext-orthogonal pairs}

\dedicatory{Dedicated to Helmut Lenzing on the occasion of his 70th birthday.}

\thanks{This research was partially supported by  DFG Schwerpunkt
  SPP 1388.\\ The second author was supported by the Research Council
  of Norway through the Storforsk-project ``Homological and geometric
  methods in algebra'' and by the grant MSM~0021620839.}

\author{Henning Krause} 
\address{Henning Krause\\ Fakult\"at f\"ur
  Mathematik\\ Universit\"at Bielefeld\\ 33501 Bielefeld\\ Germany.}
\email{hkrause@math.uni-bielefeld.de}

\author{Jan \v{S}\v{t}ov\'i\v{c}ek}
\address{Jan \v{S}\v{t}ov\'i\v{c}ek\\
Charles University in Prague\\ Faculty of Mathematics and Physics\\ Department of Algebra\\
Sokolovska 83\\ 186~75 Praha 8\\ Czech Republic.}
\email{stovicek@karlin.mff.cuni.cz}

\newtheorem{lem}{Lemma}[section]
\newtheorem{prop}[lem]{Proposition}
\newtheorem{cor}[lem]{Corollary}
\newtheorem{thm}[lem]{Theorem}

\newtheorem*{ThmA}{Theorem~A}
\newtheorem*{ThmB}{Theorem~B}

\theoremstyle{remark}
\newtheorem{rem}[lem]{Remark}

\theoremstyle{definition}
\newtheorem{exm}[lem]{Example}
\newtheorem{defn}[lem]{Definition}

\numberwithin{equation}{section}

\newcommand{\smatrix}[1]{\left[\begin{smallmatrix}#1\end{smallmatrix}\right]}

\renewcommand{\mod}{\operatorname{mod}\nolimits}

\newcommand{\Add}{\operatorname{Add}\nolimits}

\newcommand{\Id}{\operatorname{Id}\nolimits}
\newcommand{\Mod}{\operatorname{Mod}\nolimits}
\newcommand{\End}{\operatorname{End}\nolimits}
\newcommand{\Hom}{\operatorname{Hom}\nolimits}
\newcommand{\RHom}{\operatorname{\mathbf{R}Hom}\nolimits}

\renewcommand{\Im}{\operatorname{Im}\nolimits}
\newcommand{\Ker}{\operatorname{Ker}\nolimits}
\newcommand{\Coker}{\operatorname{Coker}\nolimits}

\newcommand{\gldim}{\operatorname{gldim}\nolimits}

\newcommand{\Ext}{\operatorname{Ext}\nolimits}
\newcommand{\Tor}{\operatorname{Tor}\nolimits}

\newcommand{\Fac}{\operatorname{Fac}\nolimits}
\newcommand{\Sub}{\operatorname{Sub}\nolimits}

\newcommand{\op}{\mathrm{op}}

\newcommand{\fp}{\mathrm{fp}}
\newcommand{\inj}{\mathrm{inj}}

\newcommand{\xto}{\xrightarrow}
\newcommand{\lotimes}{\otimes^{\mathbf L}}
\def\li{\varinjlim}

\def\a{\alpha}
\def\b{\beta}
\def\e{\varepsilon}

\def\g{\gamma}
\def\p{\phi}

\def\s{\sigma}

\def\Ga{\varGamma}

\def\Si{\Sigma}

\def\A{{\mathcal A}}
\def\B{{\mathcal B}}
\def\C{{\mathcal C}}
\def\D{{\mathcal D}}

\def\X{{\mathcal X}}
\def\Y{{\mathcal Y}}
\def\Z{{\mathcal Z}}
\def\T{{\mathcal T}}
\def\U{{\mathcal U}}
\def\V{{\mathcal V}}
\def\W{{\mathcal W}}

\def\bfD{{\mathbf D}}

\def\bfL{{\mathbf L}}

\def\bbN{{\mathbb N}}
\def\bbZ{{\mathbb Z}}

\begin{document}

\begin{abstract}
For the module category of a hereditary ring, the Ext-ortho\-gonal
pairs of subcategories are studied.  For each Ext-orthogonal pair that
is generated by a single module, a 5-term exact sequence is
constructed. The pairs of finite type are characterized and two
consequences for the class of hereditary rings are established:
homological epimorphisms and universal localizations coincide, and the
telescope conjecture for the derived category holds true.
However, we present examples showing that neither of these two
statements is true in general for rings of global dimension 2.
\end{abstract}

\maketitle

\section{Introduction}

In this paper, we prove the telescope conjecture for the derived
category of any hereditary ring. To achieve this, we study
Ext-orthogonal pairs of subcategories for hereditary  module categories.

\smallskip
The telescope conjecture for the derived category of a module category
is also called smashing conjecture. It is the analogue of the
telescope conjecture from stable homotopy theory which is due to
Bousfield and Ravenel \cite{B,Ra}. In each case one deals with a
compactly generated triangulated category. The conjecture then claims
that a localizing subcategory is generated by compact objects provided
it is smashing, that is, the localizing subcategory arises as the
kernel of a localization functor that preserves arbitrary coproducts
\cite{Ne1992}.  In this general form, the telescope conjecture seems
to be wide open. For the stable homotopy category, we refer to the
work of Mahowald, Ravenel, and Shick \cite{MRS} for more details. In
our case, the conjecture takes the following form and is proved in
\S\ref{se:tel}:

\begin{ThmA}
Let $A$ be a hereditary  ring. For a localizing subcategory
$\C$ of $\bfD(\Mod A)$ the following conditions are equivalent:
\begin{enumerate}
\item There exists a localization functor $L\colon\bfD(\Mod
A)\to\bfD(\Mod A)$ that preserves coproducts and such that $\C=\Ker L$.
\item The localizing subcategory $\C$ is generated by perfect
complexes.
\end{enumerate}
\end{ThmA}

For the derived category of a module category, only two results seem
to be known so far. Neeman proved the conjecture for the derived
category of a commutative noetherian ring \cite{N}, essentially by
classifying all localizing subcategories; see \cite{HPS} for a
treatment of this approach in the context of axiomatic stable homotopy
theory. On the other hand, Keller gave an explicit example of a
commutative ring where the conjecture does not hold \cite{Ke}.  In
fact, an analysis of Keller's argument \cite{Kr2005} shows that there
are such examples having global dimension $2$; see
Example~\ref{ex:tc_fail}.

The approach for hereditary rings presented here is completely different
from Neeman's.  In particular, we are working in a non-commutative
setting and without using any noetherianess assumption. The main idea here is to
exploit the very close connection between the module category and the derived
category in the hereditary case. Unfortunately, this approach cannot be
extended directly even to global dimension $2$, as mentioned above.

\smallskip
At a first glance, the telescope conjecture seems to be a rather
abstract statement about unbounded derived categories. However in the
context of a fixed hereditary ring, it turns out that smashing
localizing subcategories are in bijective correspondence to various
natural structures; see \S\ref{se:bij}:

\begin{ThmB}\label{th:ThB}
For a hereditary ring $A$ there are bijections between the following
sets:
\begin{enumerate}
\item Extension closed abelian subcategories of $\Mod A$ that are
closed under products and coproducts.
\item Extension closed abelian subcategories  of $\mod A$.
\item Homological epimorphisms $A\to B$ (up to isomorphism).
\item Universal localizations $A\to B$ (up to isomorphism).
\item Localizing subcategories of $\bfD(\Mod A)$ that are closed under products.
\item Localization functors $\bfD(\Mod A)\to\bfD(\Mod A)$ preserving
coproducts (up to natural isomorphism).
\item Thick subcategories of $\bfD^b(\mod A)$.
\end{enumerate}
\end{ThmB}

This reveals that the telescope conjecture and its proof are related to interesting
recent work by some other authors.  In \cite{Sch2007}, Schofield
describes for any hereditary ring its universal localizations in terms
of appropriate subcategories of finitely presented modules. This is a
consequence of the present work since we show that homological
epimorphisms and universal localizations coincide for any hereditary
ring; see \S\ref{se:uni-loc}. However, as we mention at the end of \S\ref{se:uni-loc},
the identification between homological epimorphisms and universal localizations
also fails already for rings of global dimension $2$.

In \cite{NS}, Nicol\'as and Saor\'in establish for a
differential graded algebra a correspondence between recollements for
its derived category and differential graded homological
epimorphisms. This correspondence specializes for a hereditary ring to
the above mentioned bijection between smashing localizing
subcategories and homological epimorphisms.

\smallskip
The link between the structures mentioned in Theorem~B is 
provided by so-called Ext-orthogonal pairs. This concept seems to be new,
but it is based on the notion of a perpendicular category which is one
of the fundamental tools for studying hereditary categories arising in
representation theory \cite{Sch1991,GL}.

Given any abelian category $\A$, we call a pair $(\X,\Y)$ of full
subcategories \emph{Ext-orthogonal} if $\X$ and $\Y$ are orthogonal to
each other with respect to the bifunctor $\coprod_{n\geq
0}\Ext^n_\A(-,-)$. This concept is the analogue of a \emph{torsion
pair} and a \emph{cotorsion pair} where one considers instead the
bifunctors $\Hom_\A(-,-)$ and $\coprod_{n> 0}\Ext^n_\A(-,-)$,
respectively \cite{D,Sa}.

Torsion and cotorsion pairs are most interesting when they are
\emph{complete}.  For a torsion pair this means that each object $M$
in $\A$ admits a short exact sequence $0\to X_M\to M\to Y^M\to 0$ with
$X_M\in\X$ and $Y^M\in\Y$. In the second case this means that each
object $M$ admits short exact sequences $0\to Y_M\to X_M\to M\to 0$
and $0\to M\to Y^M\to X^M\to 0$ with $X_M,X^M\in\X$ and
$Y_M,Y^M\in\Y$.

It turns out that there is also a reasonable notion of completeness
for Ext-orthogonal pairs. In that case each object $M$ in $\A$ admits
a 5-term exact sequence $$0\to Y_M\to X_M\to M\to Y^M \to X^M\to 0$$
with $X_M,X^M\in\X$ and $Y_M,Y^M\in\Y$. This notion of a complete
Ext-orthogonal pair is meaningful also for non-hereditary module categories,
see Example~\ref{ex:domains}.

In this work, however, we study Ext-orthogonal pairs mainly for the module category of
a hereditary ring. As already mentioned, this assumption implies a close
connection between the module category and its derived category,
which we exploit in both directions. We use Bousfield
localization functors which exist for the derived category to
establish the completeness of certain Ext-orthogonal pairs for the
module category; see \S\ref{se:ext}. On the other hand, we are able to
prove the telescope conjecture for the derived category by showing
first a similar result for Ext-orthogonal pairs; see \S\ref{se:fin}
and \S\ref{se:tel}.

Specific examples of Ext-orthogonal pairs arise in the representation
theory of finite dimensional algebras via perpendicular categories;
see \S\ref{se:exm}. Note that a perpendicular category is always a part
of an Ext-orthogonal pair.  Schofield introduced perpendicular
categories for representations of quivers \cite{Sch1991} and this fits
into our set-up because the path algebra of any quiver is
hereditary. In fact, the concept of a perpendicular category is
fundamental for studying hereditary categories arising in
representation theory \cite{GL}. It is therefore somewhat surprising
that the 5-term exact sequence for a complete Ext-orthogonal pair
seems to appear for the first time in this work.

\subsection*{Acknowledgements}
The authors would like to thank Lidia Angeleri H\"ugel and Manolo
Saor\'in for helpful discussions concerning this work.

\section{Ext-orthogonal pairs}
\label{se:ext}

Let $\A$ be an abelian category.  Given a pair of objects $X,Y\in\A$,
set $$\Ext_\A^*(X,Y)=\coprod_{n\in\bbZ}\Ext_\A^n(X,Y).$$ For a
subcategory $\C$ of $\A$ we consider its full Ext-orthogonal
subcategories
\begin{align*}
{^\perp}\C&=\{X\in\A\mid \Ext^*_\A(X,C)=0\text{ for all }C\in\C\},\\
\C^\perp&=\{Y\in\A\mid \Ext^*_\A(C,Y)=0\text{ for all }C\in\C\}.
\end{align*}
If $\C = \{X\}$ is a singleton, we write ${^\perp}X$ instead of
${^\perp}\{X\}$, and similarly with $X^\perp$.

\begin{defn}
An \emph{Ext-orthogonal pair} for $\A$ is a pair $(\X,\Y)$ of full
subcategories such that $\X^\perp=\Y$ and $\X={^\perp}\Y$.  An
Ext-orthogonal pair $(\X,\Y)$ is called \emph{complete} if there
exists for each object $M\in\A$ an exact sequence
$$\e_M\colon\:\:0\to Y_M\to X_M\to M\to Y^M\to X^M\to 0$$ with
$X_M,X^M\in\X$ and $Y_M,Y^M\in\Y$.  The pair $(\X,\Y)$ is
\emph{generated} by a subcategory $\C$ of $\A$ if $\Y=\C^\perp$.
\end{defn}

The definition can be extended to the derived category $\bfD(\A)$ of $\A$ if we put
for each pair of complexes $X,Y\in\bfD(\A)$ and $n\in\bbZ$
$$\Ext^n_\A(X,Y)=\Hom_{\bfD(\A)}(X,Y[n]).$$
Thus an \emph{Ext-orthogonal pair} for $\bfD(\A)$ is a pair $(\X,\Y)$ of
full subcategories of $\bfD(\A)$ such that $\X^\perp=\Y$ and $\X={^\perp}\Y$.

Recall that an \emph{abelian subcategory} of $\A$ is a full
subcategory $\C$ such that the category $\C$ is abelian and the
inclusion functor $\C\to\A$ is exact. Moreover, we will always assume that
an abelian subcategory $\C$ is closed under taking isomorphic objects in
the original category $\A$. Suppose $\A$ is
\emph{hereditary}, that is, $\Ext_\A^n(-,-)$ vanishes for all $n>1$. Then a
simple calculation shows that for any subcategory $\C$ of $\A$, the
subcategories $\C^\perp$ and $^{\perp}\C$ are extension closed abelian
subcategories; see \cite[Proposition~1.1]{GL}.

The following result establishes the completeness for certain
Ext-orthogonal pairs. Recall that an abelian category is a
\emph{Grothendieck category} if it has a set of generators and admits
colimits that are exact when taken over filtered categories.

\begin{thm}\label{th:perpX}
Let $\A$ be a hereditary Grothendieck category and $X$ an
object in $\A$.  Set $\Y=X^\perp$ and let $\X$ denote the smallest
extension closed abelian subcategory of $\A$ that is closed under
taking coproducts and contains $X$. Then $(\X,\Y)$ is a complete
Ext-orthogonal pair for $\A$. Thus there exists for each object
$M\in\A$ an exact sequence
$$0\to Y_M\to X_M\to M\to Y^M\to X^M\to 0$$ with
$X_M,X^M\in\X$ and $Y_M,Y^M\in\Y$. 
This sequence is natural and induces bijections
$\Hom_\A(X,X_M)\to\Hom_\A(X,M)$ and $\Hom_\A(Y^M,Y)\to\Hom_\A(M,Y)$
for all $X\in\X$ and $Y\in\Y$.
\end{thm}

The proof uses derived categories and Bousfield localization
functors. Thus we need to collect some basic facts about hereditary
abelian categories and their derived categories.

\subsection*{The derived category of a hereditary abelian category}

Let $\A$ be a hereditary abelian category and let $\bfD(\A)$ denote
its derived category.  We assume that $\A$ admits coproducts and that
the coproduct of any set of exact sequences is again exact. Thus the
category $\bfD(\A)$ admits coproducts, and for each integer $n$ these
coproducts are preserved by the functor $H^n\colon\bfD(\A)\to\A$ which
takes a complex to its cohomology in degree $n$.

It is well-known that each complex is
quasi-isomorphic to its cohomology. That is:

\begin{lem}\label{le:formal}
Given a complex $X$ in $\bfD(\A)$, there are (non-canonical) isomorphisms
\begin{equation*}
\coprod_{n\in\bbZ}(H^nX)[-n]\cong X\cong
\prod_{n\in\bbZ}(H^nX)[-n].
\end{equation*}
\end{lem}
\begin{proof}
See for instance \cite[\S1.6]{Kr}.
\end{proof}

A full subcategory $\C$ of $\bfD(\A)$ is called \emph{thick} if it
is a triangulated subcategory which is, in addition, closed under taking direct
summands. A thick subcategory is \emph{localizing} if it is closed
under taking coproducts. Note that for each full subcategory $\C$ the
subcategories $\C^\perp$ and $^\perp\C$ are thick.

To a full subcategory $\C$ of $\bfD(\A)$ we assign the full subcategory
$$H^0\C=\{M\in\A\mid M= H^0X\text{ for some }X\in\C\},$$ 
and given a full subcategory $\X$ of $\A$, we define the full subcategory
$$\bfD_\X(\A)=\{X\in\bfD(\A)\mid H^nX\in\X\text{ for all
}n\in\bbZ\}.$$ Both assignments induce mutually inverse bijections
between appropriate subcategories. 
This is a useful fact which we recall from \cite[Theorem~6.1]{Br}.

\begin{prop}\label{pr:thick-corr}
The functor $H^0\colon\bfD(\A)\to\A$ induces a bijection between the
localizing subcategories of $\bfD(\A)$ and the extension closed
abelian subcategories of $\A$ that are closed under coproducts.  The
inverse map sends a full subcategory $\X$ of $\A$ to $\bfD_\X(\A)$.
\qed
\end{prop}

\begin{rem}\label{re:thick}
The bijection in Proposition~\ref{pr:thick-corr} has an analogue for
thick subcategories.  Given any hereditary abelian category $\B$, the functor
$H^0\colon\bfD^b(\B)\to\B$ induces a bijection between the thick
subcategories of $\bfD^b(\B)$ and the extension closed abelian
subcategories of $\B$; see \cite[Theorem~5.1]{Br}.
\end{rem}

Next we extend these maps to bijections between Ext-orthogonal pairs.

\begin{prop}\label{pr:thick}
The functor $H^0\colon\bfD(\A)\to\A$ induces a bijection between the
Ext-orthogonal pairs for $\bfD(\A)$ and the Ext-orthogonal pairs for
$\A$.  The inverse map sends a pair $(\X,\Y)$ for $\A$ to
$(\bfD_\X(\A),\bfD_\Y(\A))$.
\end{prop}
\begin{proof}
First observe that for each pair of complexes $X,Y\in\bfD(\A)$, we
have $\Ext_\A^*(X,Y)=0$ if and only if $\Ext_\A^*(H^pX,H^qY)=0$ for
all $p,q\in\bbZ$. This is a consequence of Lemma~\ref{le:formal}. It
follows that $H^0$ and its inverse send Ext-orthogonal pairs to
Ext-orthogonal pairs. Each Ext-orthogonal pair is determined by its
first half, and therefore an application of
Proposition~\ref{pr:thick-corr} shows that both maps are mutually
inverse.
\end{proof}

\subsection*{Localization functors}

Let $\T$ be a triangulated category. A \emph{localization functor}
$L\colon\T\to\T$ is an exact functor that admits a natural
transformation $\eta\colon\Id_\T\to L$ such that $L\eta_X$ is
an isomorphism and $L\eta_X=\eta_{LX}$ for all objects $X\in\T$.
Basic facts about localization functors one finds, for example, in \cite[\S3]{BIK}.

\begin{prop}\label{pr:local}
Let $\A$ be a hereditary abelian category. For a full subcategory $\X$
of $\A$ the following are equivalent. 
\begin{enumerate} 
\item There exists a localization functor $L\colon\bfD(\A)\to\bfD(\A)$
such that $\Ker L=\bfD_\X(\A)$. 
\item There exists a complete Ext-orthogonal pair $(\X,\Y)$ for $\A$.
\end{enumerate}
\end{prop}
\begin{proof}
(1) $\Rightarrow$ (2): The kernel $\Ker L$ and the essential image
$\Im L$ of a localization functor $L$ form an Ext-orthogonal pair for
$\bfD(\A)$; see for instance \cite[Lemma~3.3]{BIK}.  Then it follows
from Proposition~\ref{pr:thick} that the pair $(\X,\Y)=(H^0\Ker
L,H^0\Im L)$ is Ext-orthogonal for $\A$.

The localization functor $L$ comes equipped with a natural
transformation $\eta\colon\Id_{\bfD(\A)}\to L$, and for each complex $M$ we
complete the morphism $\eta_M\colon M\to LM$ to an exact
triangle
$$\Ga M\to M\to LM\to \Ga M[1].$$ Note that $\Ga M\in\Ker L$ and
$LM\in\Im L$ since $L\eta_M$ is an isomorphism and $L$ is exact. Now
suppose that $M$ is concentrated in degree zero. Applying $H^0$ to
this triangle yields an exact sequence
$$0\to Y_M\to X_M\to M\to Y^M\to X^M\to 0$$ with $X_M,X^M\in\X$ and
$Y_M,Y^M\in\Y$.

(2) $\Rightarrow$ (1): Let $(\X,\Y)$ be an Ext-orthogonal pair for
$\A$. This pair induces an Ext-orthogonal pair
$(\bfD_\X(\A),\bfD_\Y(\A))$ for $\bfD(\A)$ by
Proposition~\ref{pr:thick}.  In order to construct a localization
functor $L\colon\bfD(\A)\to\bfD(\A)$ such that $\Ker L=\bfD_\X(\A)$,
it is sufficient to construct for each object $M$ in $\bfD(\A)$ an
exact triangle $X\to M\to Y\to X[1]$ with $X\in\bfD_\X(\A)$ and
$Y\in\bfD_\Y(\A)$. Then one defines $LM=Y$ and the morphism $M\to Y$
induces a natural transformation $\eta\colon\Id_{\bfD(\A)}\to L$
having the required properties. In view of Lemma~\ref{le:formal} it is
sufficient to assume that $M$ is a complex concentrated in degree
zero.

Suppose that $M$ admits an approximation sequence
$$\e_M\colon\:\:0\to Y_M\to X_M\to M\to Y^M\to X^M\to 0$$ with
$X_M,X^M\in\X$ and $Y_M,Y^M\in\Y$.  Let $M'$ denote the image of
$X_M\to M$ and $M''$ the image of $M\to Y^M$. Then $\e_M$ induces the
following three exact sequences
\begin{align*}
\a_M\colon\:\:&0\to M'\to M\to M''\to 0,\\
\b_M\colon\:\:&0\to Y_M\to X_M\to M'\to 0,\\
\g_M\colon\:\:&0\to M''\to Y^M\to X^M\to 0.
\end{align*}
In $\bfD(\A)$ these three exact sequence give rise to the following
commuting square
$$\xymatrix{ X^M[-2]\ar[r]^{\g_M}\ar[d]^0&M''[-1]\ar[d]^{\a_M}\\
X_M\ar[r]^{\bar\b_M}&M' }$$ where $\bar\b_M$ is the second morphism in
$\b_M$.  Commutativity of the diagram is clear since $\Hom_{\bfD(\A)}(U[-2],V)=0$
for any $U,V\in\A$. An application of the octahedral axiom shows that
this square can be extended as follows to a diagram where each row and
each column is an exact triangle.
\[\xymatrix{
X^M[-2]\ar[r]\ar[d]^0&M''[-1]\ar[r]\ar[d]&Y^M[-1]\ar[r]\ar[d]^0&X^M[-1]\ar[d]^0\\
X_M\ar[r]\ar[d]&M'\ar[r]\ar[d]&Y_M[1]\ar[r]\ar[d]&X_M[1]\ar[d]\\
X_M\oplus X^M[-1]\ar[r]\ar[d]&M\ar[r]\ar[d]&Y_M[1]\oplus
Y^M\ar[r]\ar[d]&X_M[1]\oplus X^M\ar[d]\\
X^M[-1]\ar[r]&M''\ar[r]&Y^M\ar[r]&X^M }\] The first and third column
are split exact triangles, and this explains the objects appearing in the
third row. In particular, this yields
the desired exact
triangle $X\to M\to Y\to X[1]$ with $X\in\bfD_\X(\A)$ and
$Y\in\bfD_\Y(\A)$.
\end{proof}

\begin{rem}
The proof of the implication (2) $\Rightarrow$ (1) comes as a special
case of a more general result on the existence of exact triangles with
a specified long exact sequence of cohomology objects. We refer
to work of Neeman \cite{Ne2007} for more details.
\end{rem}

Next we formulate the functorial properties of the 5-term exact
sequence constructed in Proposition~\ref{pr:local}.

\begin{lem}\label{le:exact}
Let $\A$ be an abelian category and $(\X,\Y)$ an Ext-orthogonal pair
for $\A$. Suppose there is an exact sequence
$$\e_M\colon\:\:0\to Y_M\to X_M\to M\to Y^M\to X^M\to 0$$ in $\A$ with
$X_M,X^M\in\X$ and $Y_M,Y^M\in\Y$. 
\begin{enumerate}
\item The sequence $\e_M$ induces for all $X\in\X$ and $Y\in\Y$ bijections
$\Hom_\A(X,X_M)\to\Hom_\A(X,M)$ and $\Hom_\A(Y^M,Y)\to\Hom_\A(M,Y)$.
\item Let $\e_N\colon\:\:0\to Y_N\to X_N\to N\to Y^N\to X^N\to 0$ be
an exact sequence in $\A$ with $X_N,X^N\in\X$ and $Y_N,Y^N\in\Y$. Then
each morphism $M\to N$ extends uniquely to a morphism $\e_M\to\e_N$ of
exact sequences.
\item Any exact sequence $0\to Y'\to X'\to M\to Y''\to X''\to 0$ in
$\A$ with $X',X''\in\X$ and $Y',Y''\in\Y$ is uniquely isomorphic to
$\e_M$.
\end{enumerate}
\end{lem}
\begin{proof}
We prove part (1). Then parts (2) and (3) are immediate consequences.

Fix an object $X\in\X$. The map
$\mu\colon\Hom_\A(X,X_M)\to\Hom_\A(X,M)$ is injective because
$\Hom_\A(X,Y_M)=0$. Any morphism $X\to M$ factors through the kernel
$M'$ of $M\to Y^M$ since $\Hom_\A(X,Y^M)=0$. The induced morphism
$X\to M'$ factors through $X_M\to M'$ since $\Ext_\A^1(X,Y_M)=0$. Thus
$\mu$ is surjective. The argument for the other map
$\Hom_\A(Y^M,Y)\to\Hom_\A(M,Y)$ is dual.
\end{proof}

\subsection*{Ext-orthogonal pairs for Grothendieck categories}
\label{se:groth}

Now we give the proof of Theorem~\ref{th:perpX}.  The basic idea is to
establish a localization functor for $\bfD(\A)$ and to derive the
exact approximation sequence in $\A$ by taking the cohomology of
some appropriate exact triangle as in Proposition~\ref{pr:local}.

\begin{proof}[Proof of Theorem~\ref{th:perpX}]
Let $\X$ denote the smallest extension closed abelian subcategory of
$\A$ that contains $X$ and is closed under coproducts. Then
Proposition~\ref{pr:thick-corr} implies that $\bfD_\X(\A)$ is the
smallest localizing subcategory of $\bfD(\A)$ containing $X$.  Thus
there exists a localization functor $L\colon\bfD(\A)\to\bfD(\A)$ with
$\Ker L=\bfD_\X(\A)$. This is a result which goes back to
Bousfield's work in algebraic topology, \cite{B}.
In the context of derived categories we refer to
\cite[Theorem~5.7]{AJS}.  Now apply Proposition~\ref{pr:local} to get
the 5-term exact sequence for each object $M$ in $\A$.  The properties
of this sequence follow from Lemma~\ref{le:exact}.
\end{proof}

\begin{rem}
We do not know an example of an Ext-orthogonal pair $(\X,\Y)$ for a hereditary
Grothendieck category such that the  pair $(\X,\Y)$ is not complete.
\end{rem}

Ext-orthogonal pairs naturally arise also for non-hereditary
abelian categories. Here we mention one such class of examples, but we do not know
whether or when exactly they are complete:

\begin{exm} \label{ex:groth_loc}
Let $\A$ be any Grothendieck category and $\X$ a \emph{localizing
subcategory}. That is, $\X$ is a full subcategory closed under taking
coproducts and such that for any exact sequence $0\to M'\to M\to M''\to 0$
in $\A$ we have $M\in\X$ if and only if $M',M''\in\X$.  Set
$\Y=\X^\perp$ and let $\Y_\inj$ denote the full subcategory of
injective objects of $\A$ contained in $\Y$. Then $\X={^\perp\Y_\inj}$ and therefore
$(\X,\Y)$ is an Ext-orthogonal pair for $\A$; see \cite[III.4]{Ga} for
details.
\end{exm}

\subsection*{Torsion and cotorsion pairs}\label{se:tor-cotor}

We also sketch an interpretation of an Ext-orthogonal pair in terms of
torsion and cotorsion pairs. Here, a pair $(\U,\V)$ of full
subcategories of $\A$ is called a \emph{torsion pair} if $\U$ and $\V$
are orthogonal to each other with respect to $\Hom_\A(-,-)$.
Analogously, a pair of full subcategories is a \emph{cotorsion pair}
if both categories are orthogonal to each other with respect to
$\coprod_{n>0}\Ext^n_\A(-,-)$.

Let $\A$ be an abelian category and $(\X,\Y)$ an Ext-orthogonal pair.
The subcategory $\X$ generates a torsion pair $(\X_0,\Y_0)$ and a
cotorsion pair $(\X_1,\Y_1)$ for $\A$, if one defines the
corresponding full subcategories of $\A$ as follows:
\begin{align*}
\Y_0&=\{Y\in\A\mid\Hom_\A(X,Y)=0\text{ for all }X\in\X\},\\
\X_0&=\{X\in\A\mid\Hom_\A(X,Y)=0\text{ for all }Y\in\Y_0\},\\
\Y_1&=\{Y\in\A\mid\Ext^n_\A(X,Y)=0\text{ for all }X\in\X,\,n>0\},\\
\X_1&=\{X\in\A\mid\Ext^n_\A(X,Y)=0\text{ for all }Y\in\Y_1,\,n>0\}.
\end{align*}
Note that $\X=\X_0\cap\X_1$ and $\Y=\Y_0\cap\Y_1$.  In particular, one
recovers the pair $(\X,\Y)$ from $(\X_0,\Y_0)$ and $(\X_1,\Y_1)$.

Suppose an object $M\in\A$ admits an approximation sequence
$$\e_M\colon\:\:0\to Y_M\to X_M\to M\to Y^M\to X^M\to 0$$ with
$X_M,X^M\in\X$ and $Y_M,Y^M\in\Y$. We give the following
interpretation of this sequence.  Let $M'$ denote the image of $X_M\to
M$ and $M''$ the image of $M\to Y^M$. Then there are three short exact
sequences:
\begin{align*}
\a_M\colon\:\:&0\to M'\to M\to M''\to 0,\\
\b_M\colon\:\:&0\to Y_M\to X_M\to M'\to 0,\\
\g_M\colon\:\:&0\to M''\to Y^M\to X^M\to 0.
\end{align*}
The sequence $\a_M$ is the approximation sequence of $M$ with respect
to the torsion pair $(\X_0,\Y_0)$, that is, $M'\in\X_0$ and
$M''\in\Y_0$. On the other hand, $\b_M$ and $\g_M$ are approximation
sequences of $M'$ and $M''$ respectively, with respect to the
cotorsion pair $(\X_1,\Y_1)$, that is, $X_M,X^M\in\X_1$ and
$Y_M,Y^M\in\Y_1$. Thus the 5-term exact sequence $\e_M$ is obtained by
splicing together three short exact approximation sequences.
 
Suppose finally that the Ext-orthogonal pair $(\X,\Y)$ is complete.
It is not hard to see that then the
associated torsion pair $(\X_0,\Y_0)$ has an explicit description:
we have $\X_0=\Fac\X$ and $\Y_0=\Sub \Y$, where
$$\Fac\X=\{X/U\mid U\subseteq X,\,X\in\X\}\quad\text{and}\quad
\Sub\Y=\{U\mid U\subseteq Y,\,Y\in\Y\}.$$

\section{Homological epimorphisms}

From now on we will study Ext-orthogonal pairs only for module categories. Thus
we fix a ring $A$ and denote by $\Mod A$ the category of (right)
$A$-modules. The full subcategory formed by all finitely presented
$A$-modules is denoted by $\mod A$. 

Most of our results require the ring $A$ to be (right) \emph{hereditary}. This
means the category of $A$-modules is hereditary, that is,
$\Ext_A^n(-,-)$ vanishes for all $n>1$.

We are going to show that Ext-orthogonal pairs
for module categories over hereditary rings are
closely related to homological epimorphisms.  Recall that a ring
homomorphism $A\to B$ is a \emph{homological epimorphism} if
$$B\otimes_AB\cong B\quad\text{and}\quad\Tor_n^A(B,B)=0\quad\text{for
all}\quad n>0,$$ or equivalently, if restriction induces isomorphisms
$$\Ext_B^*(X,Y)\xto{\sim}\Ext_A^*(X,Y)$$ for all $B$-modules $X,Y$;
see \cite{GL} for details. The first observation is that every homological
epimorphism naturally induces two complete Ext-orthogonal pairs:

\begin{prop}\label{pr:hom-epi1}
Let $A$ be a hereditary ring and $f\colon A\to B$ a homological
epimorphism. Denote by $\Y$ the category of $A$-modules which are
restrictions of modules over $B$. Set $\X={^\perp\Y}$ and
$\Y^\perp=\Z$.  Then $(\X,\Y)$ and $(\Y,\Z)$ are complete
Ext-orthogonal pairs for $\Mod A$ with $\Y=(\Ker f\oplus\Coker
f)^\perp$ and $\Z=B^\perp$.
\end{prop}
\begin{proof}
We wish to apply Theorem~\ref{th:perpX} which provides a construction for
complete Ext-orthogonal pairs.

First observe that $\Y$ is the smallest extension closed abelian
subcategory of $\Mod A$ closed under coproducts and containing
$B$. This yields $\Z=B^\perp$.

Next we show that $\Y=(\Ker f\oplus\Coker f)^\perp$.  In fact, an
$A$-module $Y$ is the restriction of a $B$-module if and only if $f$
induces an isomorphism $\Hom_A(B,Y)\to\Hom_A(A,Y)$. Using the
assumptions on $A$ and $f$, a simple calculation shows that this
implies $\Y=(\Ker f\oplus\Coker f)^\perp$.

It remains to apply Theorem~\ref{th:perpX}. Thus $(\X,\Y)$ and
$(\Y,\Z)$ are complete Ext-orthogonal pairs.
\end{proof}

Now we use a crucial theorem of Gabriel and de la Pe\~na.
It identifies, only by their closure properties, the
full subcategories of a module category $\Mod A$ that arise as the
images of the restriction functors $\Mod B\to\Mod A$ for ring
epimorphisms $A\to B$. In our version, we identify in a similar way the
essential images of the restriction functors of homological epimorphisms,
provided $A$ is hereditary.

\begin{prop}\label{pr:hom-epi2}
Let $A$ be a hereditary ring and $\Y$ an extension closed abelian
subcategory of $\Mod A$ that is closed under taking products and
coproducts. Then there exists a homological epimorphism $f\colon A\to
B$ such that the restriction functor $\Mod B\to\Mod A$ induces an
equivalence $\Mod B\xto{\sim}\Y$. 
\end{prop}

\begin{proof}
It follows from \cite[Theorem~1.2]{GP} that there exists a ring
epimorphism $f\colon A\to B$ such that the restriction functor $\Mod
B\to\Mod A$ induces an equivalence $\Mod B\xto{\sim}\Y$. To be more
specific, one constructs a left adjoint $F\colon \Mod A\to \Y$ for the
inclusion $\Y\to \Mod A$. Then $FA$ is a small projective generator
for $\Y$, because $A$ has this property for $\Mod A$ and the inclusion
of $\Y$ is an exact functor that preserves coproducts.  Thus one
takes for $f$ the induced map $A\cong\End_A(A)\to\End_A(FA)$.

We claim that restriction via $f$ induces an isomorphism
\begin{equation*}
\Ext_B^n(X,Y)\xto{\sim}\Ext_A^n(X,Y)
\end{equation*} 
for all $B$-modules $X,Y$ and
all $n\geq 0$. This is clear for $n=0,1$ since $\Y$ is extension
closed. On the other hand, the isomorphism for $n=1$ implies that
$\Ext_B^1(X,-)$ is right exact since $A$ is hereditary. It follows
that $B$ is hereditary and $\Ext_B^n(-,-)$ vanishes for all $n>1$. 
\end{proof}

We get as an immediate consequence that any class
$\mathcal{Y}$ satisfying the assumptions of Proposition~\ref{pr:hom-epi2}
belongs to two complete cotorsion pairs. In order to obtain more information
about the corresponding 5-term approximation sequences, we prefer, however, to
postpone this corollary after the following lemma:

\begin{lem}\label{le:hom-epi}
Let $A\to B$ be a homological epimorphism and denote by $\Y$ the
category of $A$-modules which are restrictions of modules over $B$.
\begin{enumerate}
\item The functor $\bfD(\Mod
A)\to\bfD(\Mod A)$ sending a complex $X$ to $X\otimes_A^\bfL B$ is a
localization functor with essential image equal to $\bfD_\Y(\Mod A)$.
\item The functor $\bfD(\Mod A)\to\bfD(\Mod A)$ sending a complex $X$
to the cone (which is in this case functorial) of the natural morphism $\RHom_A(B,X)\to X$ is a
localization functor with kernel equal to $\bfD_\Y(\Mod A)$.
\end{enumerate}
\end{lem}
\begin{proof}
Restriction along $f\colon A\to B$ identifies $\Mod B$ with $\Y$. The
functor induces an isomorphism
\begin{equation*}
\Ext_B^n(X,Y)\xto{\sim}\Ext_A^n(X,Y)
\end{equation*} 
for all $B$-modules $X,Y$ and all $n\geq 0$, because $f$ is a
homological epimorphism. This isomorphism implies that the induced
functor $f_*\colon\bfD(\Mod B)\to\bfD(\Mod A)$ is fully faithful with
essential image $\bfD_\Y(\Mod A)$. Moreover, $f_*$ is naturally isomorphic
to both $\RHom_B({_A}B,-)$ and $-\otimes^\bfL_BB_A$. It follows that:

(1) The functor $f_*$ admits a left adjoint $f^*=-\otimes^\bfL_AB$ and
we therefore have a localization functor $L\colon \bfD(\Mod
A)\to\bfD(\Mod A)$ sending a complex $X$ to $f_*f^*(X)$; see
\cite[Lemma~3.1]{BIK}.  It remains to note that the essential images of
$L$ and  $f_*$ coincide.

(2) The functor $f_*$ admits a right adjoint $f^!=\RHom_A(B,-)$ and we
therefore have a colocalization functor $\Ga\colon \bfD(\Mod
A)\to\bfD(\Mod A)$ sending a complex $X$ to $f_*f^!(X)$.  Note that
the adjunction morphism $\Ga X\to X$ is an isomorphism if and only if
$X$ belongs to $\bfD_\Y(\Mod A)$.  Completing $\Ga X\to X$ to a
triangle yields a well defined localization functor $\bfD(\Mod
B)\to\bfD(\Mod A)$ with kernel $\bfD_\Y(\Mod A)$; see
\cite[Lemma~3.3]{BIK}.
\end{proof}

Now we state the above mentioned immediate consequence of
Propositions~\ref{pr:hom-epi1} and \ref{pr:hom-epi2}, but with an alternative
and more explicit proof.

\begin{cor}\label{co:perp}
Let $A$ be a hereditary ring and $\Y$ an extension closed abelian
subcategory of $\Mod A$ that is closed under taking products and
coproducts. Set $\X={^\perp} \Y$ and $\Z=\Y^\perp$.  Then $(\X,\Y)$
and $(\Y,\Z)$ are both complete Ext-orthogonal pairs.
\end{cor}
\begin{proof}
There exists a homological epimorphism $f\colon A\to B$ such that
restriction identifies $\Mod B$ with $\Y$; see
Proposition~\ref{pr:hom-epi2}.  Then Lemma~\ref{le:hom-epi} produces two
localization functors $L_1,L_2\colon \bfD(\Mod A)\to\bfD(\Mod A)$ with
$\Im L_1=\bfD_\Y(\Mod A)=\Ker L_2$. Thus
$$\Ker L_1={^\perp}(\Im L_1)=\bfD_\X(\Mod A)\quad \text{and}\quad \Im
L_2=(\Ker L_2)^\perp=\bfD_\Z(\Mod A),$$ where in both cases the first
equality follows from \cite[Lemma~3.3]{BIK} and the second from
Proposition~\ref{pr:thick}.  It remains to apply
Proposition~\ref{pr:local} which yields in both cases for each $A$-module the
desired 5-term exact sequence.
\end{proof}

\begin{rem}\label{re:5seq}
The proof of Lemma~\ref{le:hom-epi} and Corollary~\ref{co:perp} yields for any $A$-module $M$ an
explicit description of some terms of the 5-term exact sequence
$\e_M$, using the homological epimorphism $A\to B$. In the first case,
we have
$$\e_M\colon\:\: 0\to \Tor_1^A(M,B)\to X_M\to M\to M\otimes_AB\to X^M\to 0,$$
and in the second case, we have
$$\e_M\colon\:\: 0\to Z_M\to \Hom_A(B,M)\to M\to Z^M\to\Ext_A^1(B,M)\to 0.$$
\end{rem}

We also mention another consequence of the above discussion, which is immediately
implied by Corollary~\ref{co:perp}. It reflects the fact that given a homological
epimorphism $A\to B$ and the fully faithful functor
$f_*\colon \bfD(\Mod B) \to \bfD(\Mod A)$ having both a left and a right adjoint, there
exists a corresponding recollement of the derived category $\bfD(\Mod A)$; see
\cite[\S4.13]{Kr2008}.

\begin{cor}
Let $A$ be a hereditary ring and $(\X,\Y)$ an Ext-orthogonal pair
for the category of $A$-modules.
\begin{enumerate}
\item There is an Ext-orthogonal pair $(\W,\X)$ if and only if
$\X$ is closed under products.
\item There is an Ext-orthogonal pair $(\Y,\Z)$ if and only if
$\Y$ is closed under coproducts.
\end{enumerate}
\end{cor}

\section{Examples}\label{se:exm}

We present a number of examples of Ext-orthogonal pairs which
illustrate the results of this work. The first example is classical
and provides one of the motivations for studying perpendicular
categories in representation theory of finite dimensional algebras.
We refer to Schofield's work \cite{Sch1986,Sch1991} which contains some
explicit calculations; see also \cite{GL,H}.

\begin{exm} 
Let $A$ be a finite dimensional hereditary algebra over a field $k$
and $X$ a finite dimensional $A$-module.  Then $X^\perp=\Y$ identifies
via a homological epimorphism $A\to B$ with the category of modules
over a $k$-algebra $B$ and this yields a complete Ext-orthogonal pair
$(\X,\Y)$. If $X$ is \emph{exceptional}, that is, $\Ext^1_A(X,X)=0$,
then $B$ is finite dimensional (see the proposition below) and often
can be constructed explicitly. We refer to~\cite{Sch1986} for particular examples.
Note that in this case for each finite dimensional $A$-module $M$ the
corresponding 5-term exact sequence $\e_M$ consists of finite
dimensional modules. Moreover, the category $\X$ is equivalent to the
module category of another finite dimensional algebra.  We do not know
of a criterion on $X$ that characterizes the fact that $B$ is finite
dimensional; see however the following proposition.
\end{exm}

\begin{prop}\label{pr:except}
Let $A$ be a finite dimensional hereditary algebra over a field $k$ and
$(\X,\Y)$ a complete Ext-orthogonal pair such that $\Y$ is closed
under coproducts. Fix a homological epimorphism $A\to B$ inducing an
equivalence $\Mod B\xto{\sim}\Y$. Then the following are equivalent.
\begin{enumerate}
\item There exists an exceptional module $X\in\mod A$ such that $\Y=X^\perp$.
\item The algebra $B$ is finite dimensional over $k$. 
\item For each $M\in\mod A$, the 5-term exact sequence $\e_M$ belongs to $\mod A$.
\end{enumerate}
\end{prop}
\begin{proof}
(1) $\Rightarrow$ (2): This follows, for example, from \cite[Proposition~3.2]{GL}.

(2) $\Rightarrow$ (3): This follows from Remark~\ref{re:5seq}.

(3) $\Rightarrow$ (1): Let $\X_\fp=\X\cap\mod A$ and
$\Y_\fp=\Y\cap\mod A$. The assumption on $(\X,\Y)$ implies that
$(\X_\fp,\Y_\fp)$ is a complete Ext-orthogonal pair for $\mod
A$. Moreover, every object in $\X$ is a filtered colimit of objects in
$\X_\fp$. To see this, we first express $X$ as a filtered colimit
$\varinjlim M_i$ of finitely presented modules. Then, using the forthcoming
Lemma~\ref{le:seq}(2), we see that $\varepsilon_X = \varinjlim \varepsilon_{M_i}$,
from which it easily follows that $X \cong \varinjlim X_{M_i}$.
Now choose an injective cogenerator
$Q$ in $\mod A$ and let $X=X_Q$ be the module from the 5-term exact
sequence $\e_Q$. This module is the image of $Q$ under a right adjoint
of the inclusion $\X_\fp\to\mod A$. Note that a right adjoint of an
exact functor preserves injectivity.  It follows that $X$ is an
exceptional object and that $\X_\fp$ is the smallest extension closed
abelian subcategory of $\mod A$ containing $X$. Thus
$X^\perp=\X_\fp^\perp=\X^\perp=\Y$, using the fact that
$\X = \varinjlim \X_\fp$.
\end{proof}

As a special case, any finitely generated projective module generates an Ext-orthog\-onal
pair that can be described explicitly; see \cite[\S5]{GL}. For cyclic projective modules,
this is discussed in more generality in the following example.

\begin{exm}
Let $A$ be a hereditary ring and $e^2=e\in A$ an idempotent.  Let $\X$
denote the category of $A$-modules $M$ such that the natural map
$Me\otimes_{eAe}eA\to M$ is an isomorphism, and let
$\Y=eA^\perp=\{M\in\Mod A\mid Me=0\}$.  Thus $-\otimes_{eAe} eA$
identifies $\Mod eAe$ with $\X$ and restriction via $A\to A/AeA$
identifies $\Mod A/AeA$ with $\Y$.  Then $(\X,\Y)$ is a complete
Ext-orthogonal pair for $\Mod A$, and for each $A$-module $M$ the
5-term exact sequence $\e_M$ is of the form
$$0\to \Tor_1^A(M,A/AeA)\to Me\otimes_{eAe}eA\to M\to
M\otimes_A A/AeA\to 0 \to 0.$$
\end{exm}

The next example\footnote{The first author is grateful to Lidia
Angeleri H\"ugel for suggesting this example.} arises from the work of
Reiten and Ringel on infinite dimensional representations of canonical
algebras; see \cite{RR} which is our reference for all concepts and
results in the following discussion.  Note that these algebras are not
necessarily hereditary.  The example shows the interplay between
Ext-orthogonal pairs and (co)torsion pairs.

\begin{exm}
Let $A$ be a finite dimensional canonical algebra over a field
$k$. Take for example a tame hereditary algebra, or, more
specifically, the Kronecker algebra $\smatrix{k&k^2\\ 0&k}$. For such
algebras, there is the concept of a \emph{separating tubular family}.
We fix such a family and denote by $\T$ the category of finite
dimensional modules belonging to this family. There is also a
particular \emph{generic module} over $A$ which depends in some cases
on the choice of the tubular family; it is denoted by $G$.  Then the full
subcategory $\X=\li\T$ consisting of all filtered colimits of modules
in $\T$ and the full subcategory $\Y=\Add G$ consisting of all
coproducts of copies of $G$ form an Ext-orthogonal pair $(\X,\Y)$ for
$\Mod A$.  Note that the endomorphism ring $D=\End_A(G)$ of $G$ is a
division ring and that the canonical map $A\to B$ with $B=\End_D(G)$
is a homological epimorphism which induces an equivalence $\Mod
B\xto{\sim}\Y$. In the particular case of the Kronecker algebra
$A = \smatrix{k&k^2\\ 0&k}$, a direct computation shows that
$B = M_2\big(k(x)\big)$.

The category of $A$-modules which are generated by $\T$ and the
category of $A$-modules which are cogenerated by $G$ form a torsion
pair $(\Fac\X,\Sub\Y)$ for $\Mod A$ which equals the torsion pair
$(\X_0,\Y_0)$ generated by $\X$. On the other hand, let $\C$ denote
the category of $A$-modules which are cogenerated by $\X$, and let
$\D$ denote the category of $A$-modules $M$ satisfying
$\Hom_A(M,\T)=0$.  Then the pair $(\C,\D)$ forms a cotorsion pair for
$\Mod A$ which identifies with the cotorsion pair $(\X_1,\Y_1)$
generated by $\X$.

If $A$ is hereditary, then the Ext-orthogonal pair $(\X,\Y)$ is
complete by Corollary~\ref{co:perp}; see also Remark~\ref{re:5seq} for
an explicit description of the 5-term approximation sequence $\e_M$
for each $A$-module $M$. Alternatively, one obtains the sequence
$\e_M$ by splicing together appropriate approximation sequences which
arise from $(\X_0,\Y_0)$ and $(\X_1,\Y_1)$.
\end{exm}

The following example of an Ext-orthogonal pair arises from a
localizing subcategory; it is a specialization of Example~\ref{ex:groth_loc}
and provides a simple (and not necessarily hereditary) model
for the previous example.

\begin{exm} \label{ex:domains}
Let $A$ be an integral domain with quotient field $Q$. Let $\X$ denote
the category of torsion modules and $\Y$ the category of torsion free
divisible modules. Note that the modules in $\Y$ are precisely the
coproducts of copies of $Q$. Then $(\X,\Y)$ is a complete
Ext-orthogonal pair for $\Mod A$, and for each $A$-module $M$ the
5-term exact sequence $\e_M$ is of the form
$$ 0 \to 0\to tM \to M\to
M\otimes_A Q\to \bar M\to 0.$$
\end{exm}

We conclude the section by showing that there are examples of abelian
categories that admit only trivial Ext-orthogonal pairs.  

\begin{exm}
Let $A$ be a local artinian ring and set $\A=\Mod A$. Then
$\Hom_A(X,Y)\neq 0$ for any pair $X,Y$ of non-zero $A$-modules. This is
because the unique (up to isomorphism) simple module $S$ is a submodule
of $Y$ and a factor of $X$.
Thus if $(\X,\Y)$ is an Ext-orthogonal pair for $\A$, then $\X=\A$ or
$\Y=\A$.
\end{exm}

\section{Ext-orthogonal pairs of finite type}\label{se:fin}

At this point, we use the results from \S3 to characterize for hereditary rings
the Ext-orthogonal pairs of \emph{finite type}. Those are, by definition, the
Ext-orthogonal pairs generated by a set of finitely presented modules.

\begin{thm}\label{th:ext}
Let $A$ be a hereditary  ring and $(\X,\Y)$ an
Ext-orthogonal pair for the module category of $A$.  Then the
following are equivalent.
\begin{enumerate}
\item The subcategory $\Y$ is closed under taking coproducts.
\item Every module in $\X$ is a filtered colimit of
finitely presented modules from $\X$.
\item There exists a category $\C$ of finitely presented modules such that
$\C^\perp=\Y$.
\end{enumerate}
\end{thm}

We need some preparations for the proof of this result. The first
lemma is a slight modification of \cite[Proposition~2.1]{BET}.

\begin{lem}\label{le:ext1perp}
Let $A$ be a ring and $\Y$ a full subcategory of its module
category. Denote by $\X$ the category  of $A$-modules $X$ of projective
dimension at most $1$ satisfying $\Ext^1_A(X,Y)=0$ for all
$Y\in\Y$. Then any module in $\X$ is a filtered colimit of finitely
presented modules from $\X$.
\end{lem}
\begin{proof}
Let $X\in\X$.  Choose an exact sequence $0\to P\xto{\p} Q\to X\to 0$
such that $P$ is free and $Q$ is projective.  Note that
$\Ext^1_A(X,Y)=0$ implies that every morphism $P\to
Y$ factors through $\p$. The commuting diagrams of $A$-module
morphisms
$$\xymatrix{0\ar[r]&P_i\ar[r]^{\p_i}\ar[d]&Q_i\ar[d]\ar[r]&X_i\ar[d]\ar[r]&0\\
0\ar[r]&P\ar[r]^\p&Q\ar[r]&X\ar[r]&0 }$$ with $P_i$ and $Q_i$ finitely
generated projective form a filtered system of exact sequences such
that $\li\p_i=\p$. Note that $P$ is a filtered colimit of its finitely
generated direct summands since $P$ is free. Thus there is a cofinal
subsystem such that each morphism $P_i\to P$ is a split monomorphism.
Therefore we may without loss of generality assume that each morphism
$P_i\to P$ is a split monomorphism.

Clearly $\li X_i=X$, and it remains to prove that $\Ext_A^1(X_i,\Y)=0$
for all $i$. This is equivalent to showing that each morphism
$\mu\colon P_i\to Y$ with $Y\in\Y$ factors through $\p_i$. For this,
we first factor each such $\mu$ through the split monomorphism $P_i
\to P$, then through $\p$, and finally compose the morphism $Q \to Y$
which we have obtained with the morphism $Q_i \to Q$. The result is a
morphism $\nu\colon Q_i \to Y$ such that $\nu \p_i = \mu$, as desired.
\end{proof}

The second lemma establishes some necessary properties of the 5-term sequences.

\begin{lem}\label{le:seq}
Let $A$ be a hereditary ring and $(\X,\Y)$ a complete Ext-orthogonal
pair for $\Mod A$.  Let $M$ be an $A$-module and $\e_M$ the
corresponding 5-term exact sequence.
\begin{enumerate}
\item If  $\Ext_A^1(M,\Y)=0$, then $Y_M=0$.
\item Suppose that $\Y$ is closed under coproducts and let $M=\li M_i$
be a filtered colimit of $A$-modules $M_i$. Then $\e_M=\li\e_{M_i}$.
\end{enumerate}
\end{lem}
\begin{proof}
We use the uniqueness of the 5-term exact sequences guaranteed by Lemma~\ref{le:exact}.
If  $\Ext_A^1(M,\Y)=0$, then the image of
the morphism $X_M\to M$ belongs to $\X$. Thus $X_M\to M$ is a monomorphism since $\e_M$
is unique, and this yields (1).

To prove (2), one uses that $\X$ and $\Y$ are closed under taking
colimits and that taking filtered colimits is exact. Thus
$\li\e_{M_i}$ is an exact sequence with middle term $M$ and all other
terms in $\X$ or $\Y$. Now the uniqueness of $\e_M$ implies that
$\e_M=\li\e_{M_i}$.
\end{proof}

Finally, the following lemma is needed for hereditary rings which are not noetherian.

\begin{lem}\label{le:subfp}
Let $M$ be a finitely presented module over a hereditary ring and
$N\subseteq M$ any submodule. Then $N$ is a direct sum of finitely
presented modules.
\end{lem}
\begin{proof}
We combine two results. Over a hereditary ring, any submodule of a
finitely presented module is a direct sum of a finitely presented
module and a projective module; see \cite[Theorem~5.1.6]{C}. In
addition, one uses that any projective module is a direct sum of
finitely generated projective modules; see \cite{Al}.
\end{proof}

\begin{proof}[Proof of Theorem~\ref{th:ext}]
(1) $\Rightarrow$ (2): Suppose that $\Y$ is closed under taking
coproducts.  We apply Corollary~\ref{co:perp} to obtain for each
module $M$ the natural exact sequence $\e_M$. Here note that we a priori
did not assume completeness of $(\X,\Y)$. Now suppose that
$M$ belongs to $\X$. Then one can write $M=\li M_i$ as a filtered colimit
of finitely presented modules with $\Ext_A^1(M_i,\Y)=0$ for all $i$;
see Lemma~\ref{le:ext1perp}.  Next we apply Lemma~\ref{le:seq}. Thus
$$\li X_{M_i}\xto{\sim}X_M\xto{\sim}M,$$ and each $X_{M_i}$ is a
submodule of the finitely presented module $M_i$. Finally, each
$X_{M_i}$ is a filtered colimit of finitely presented direct summands
by Lemma~\ref{le:subfp}. Thus $M$ is a filtered colimit of finitely
presented modules from $\X$.

(2) $\Rightarrow$ (3): Let $\X_\fp$ denote the full subcategory that
is formed by all finitely presented modules in $\X$.  Observe that
$^\perp Y$ is closed under taking colimits for each module $Y$,
because $^\perp Y$ is closed under taking coproducts and cokernels.
Thus $\X_\fp^\perp=\X^\perp=\Y$ provided that $\X=\li\X_\fp$.

(3) $\Rightarrow$ (1): Use that for each finitely presented $A$-module $X$,
the functor $\Ext_A^*(X,-)$ preserves all coproducts.
\end{proof}

Note that Theorem~\ref{th:ext} gives rise to a bijection between extension closed
abelian subcategories of finitely presented modules and Ext-orthogonal
pairs of finite type. We will state this explicitly in \S8, but we in fact
prove it here by the following proposition.

\begin{prop}
\label{pr:ext}
Let $A$ be a hereditary ring and $\C$ a category of finitely presented
$A$-modules. Then $^\perp(\C^\perp)\cap\mod A$ equals the smallest
extension closed abelian subcategory of $\mod A$ containing $\C$.
\end{prop} 
\begin{proof}
Let $\D$ denote the smallest extension closed abelian subcategory of
$\mod A$ containing $\C$. We claim that the category $\li\D$ which is
formed by all filtered colimits of modules in $\D$ is an extension
closed abelian subcategory of $\Mod A$. 

Assume for the moment that the claim holds. Then Theorem~\ref{th:perpX}
implies that $\X={^\perp(\C^\perp)}$ equals the smallest extension closed abelian
subcategory of $\Mod A$ closed under coproducts and containing $\C$.
Our claim then implies $\X=\li\D$, so $\X\cap\mod A=\D$ and we are finished.

Therefore, it only remains to prove the claim. First observe that every
morphism in $\li\D$ can be
written as a filtered colimit of morphisms in $\D$. Using that taking
filtered colimits is exact, it follows immediately that $\li\D$ is
closed under kernels and cokernels in $\Mod A$. 

It remains to show that $\li\D$ is closed under extensions. To this
end let $\eta\colon 0\to L\to M\to N\to 0$ be an exact sequence with $L$ and $N$
in $\li\D$.  We can without loss of generality assume that $N$ belongs
to $\D$, because otherwise the sequence $\eta$ is a filtered colimit of the pull-back
exact sequences with the last terms in $\D$. Next we choose a morphism
$\p\colon M'\to M$ with $M'$ finitely presented. All we need to do now is to show that
$\p$ factors through an object in $\D$; see \cite{L}. We may, moreover, assume that the
composite of $\p$ with $M\to N$ is an epimorphism. This is because otherwise
we can take an epimorphism $P \to N$ with $P$ finitely generated projective, factor
it through $M \to N$, and
replace $\p$ by $\p'\colon M' \oplus P \to M$. Finally, denote by $L'$ the
kernel of $\p$, which is necessarily a finitely presented module.
The induced map $L'\to L$ then factors through an object $L''$
in $\D$ since $L$ belongs to $\li\D$.
Forming the push-out exact sequence of $0\to L'\to M'\to N\to 0$ along
the morphism $L'\to L''$ gives an exact sequence $0\to L''\to M''\to
N\to 0$. Now $\p$ factors through $M''$ which belongs to $\D$.
\end{proof}

\section{Universal localizations}
\label{se:uni-loc}

A ring homomorphism $A\to B$ is called a \emph{universal localization}
if there exists a set $\Si$ of morphisms between finitely generated
projective $A$-modules such that
\begin{enumerate}
\item $\s\otimes_AB$ is an isomorphism of $B$-modules for all $\s\in\Si$, and
\item every ring homomorphism $A\to B'$ such that $\s\otimes_AB'$ is
an isomorphism of $B$-modules for all $\s\in\Si$ factors uniquely
through $A\to B$.
\end{enumerate}

Let $A$ be a ring and $\Si$ a set of morphisms between finitely
generated projective $A$-modules. Then there exists a universal
localization inverting $\Si$ and this is unique up to a unique
isomorphism; see \cite{S} for details. The universal localization is
denoted by $A\to A_\Si$ and restriction identifies $\Mod A_\Si$ with
the full subcategory consisting of all $A$-modules $M$ such that
$\Hom_A(\s,M)$ is an isomorphism for all $\s\in\Si$. Note that
$\Hom_A(\s,M)$ is an isomorphism if and only if $M$ belongs to
$\{\Ker\s,\Coker\s\}^\perp$, provided that $A$ is hereditary.
The main result of this section is then the following theorem.

\begin{thm}\label{th:epi}
Let $A$ be a hereditary  ring. A ring homomorphism $f\colon
A\to B$ is a homological epimorphism if and only if $f$ is a universal
localization.
\end{thm}
\begin{proof}
Suppose first that $f\colon A\to B$ is a homological epimorphism. This
gives rise to an Ext-orthogonal pair $(\X,\Y)$ for $\Mod A$, if we
identify $\Mod B$ with a full subcategory $\Y$ of $\Mod A$; see
Proposition~\ref{pr:hom-epi1}. Let $\X_\fp$ denote the full subcategory
that is formed by all finitely presented modules in $\X$.  It follows
from Theorem~\ref{th:ext} that $\X_\fp^\perp=\Y$. Now fix for each
$X\in\X_\fp$ an exact sequence $$0\to P_X\xto{\s_X}Q_X\to X\to 0$$
such that $P_X$ and $Q_X$ are finitely generated projective, and let
$\Si=\{\s_X\mid X\in\X_\fp\}$. Then $$\Mod B=\X_\fp^\perp=\Mod
A_\Si.$$ Therefore, $f\colon A\to B$ is a universal localization, since
$\X_\fp^\perp$ determines the corresponding ring epimorphism
uniquely up to isomorphism, see the proof of Proposition~\ref{pr:hom-epi2}.

Now suppose $f\colon A\to B$ is a universal localization. Then
restriction identifies the category of $B$-modules with a full extension
closed subcategory of $\Mod A$. Thus we have induced isomorphisms
$$\Ext_B^*(X,Y)\xto{\sim}\Ext_A^*(X,Y)$$ for all $B$-modules $X,Y$,
since $A$ is hereditary. It follows that $f$ is a homological epimorphism.
\end{proof}

\begin{rem}
Neither implication in Theorem~\ref{th:epi} is true if one drops the
assumption on the ring $A$ to be hereditary, not even if the global
dimension is $2$. In \cite{Ke}, Keller gives an example of a B\'ezout
domain $A$ and a non-zero ideal $I$ such that the canonical map
$A\to A/I$ is a homological epimorphism, but any map $\s$ between
finitely generated projective $A$-modules needs to be invertible
if $\s\otimes_AA/I$ is invertible. We refine the construction so that
$\gldim A = 2$, see Example~\ref{ex:tc_fail}. On the other
hand, Neeman, Ranicki, and Schofield use finite dimensional algebras
to construct in \cite{NRS} examples of universal localizations that
are not homological epimorphisms. They are also able to construct such
examples of global dimension $2$, see~\cite[Remark 2.13]{NRS}.
\end{rem}

\section{The telescope conjecture}\label{se:tel}

Now we are ready to state and prove an extended version of Theorem~A after recalling
the necessary notions.

Let $A$ be a ring.  A complex of $A$-modules is called \emph{perfect}
if it is isomorphic to a bounded complex of finitely generated
projective modules.  Note that a complex $X$ is perfect if and only if
the functor $\Hom_{\bfD(\Mod A)}(X,-)$ preserves coproducts. One
direction of this statement is easy to prove since $\Hom_{\bfD(\Mod
A)}(A,-)$ preserves coproducts and every perfect complex is finitely
built from $A$. The converse follows from \cite[Lemma~2.2]{Ne1992} and
\cite[Proposition~3.4]{BN}. Recall also that a localizing subcategory
$\C$ of $\bfD(\Mod A)$ is \emph{generated by perfect complexes} if
$\C$ admits no proper localizing subcategory containing all perfect
complexes from $\C$.

\begin{thm}\label{th:TC}
Let $A$ be a hereditary  ring. For a localizing subcategory
$\C$ of $\bfD(\Mod A)$ the following conditions are equivalent:
\begin{enumerate}
\item There exists a localization functor $L\colon\bfD(\Mod
A)\to\bfD(\Mod A)$ that  preserves coproducts and such that $\C=\Ker L$.
\item The localizing subcategory $\C$ is generated by perfect
complexes.
\item There exists a localizing subcategory $\D$ of $\bfD(\Mod
A)$ that is closed under products such that  $\C={^\perp}\D$.
\end{enumerate}
\end{thm}
\begin{proof}
(1) $\Rightarrow$ (2): The kernel $\Ker L$ and the essential image
$\Im L$ of a localization functor $L$ form an Ext-orthogonal pair for
$\bfD(\Mod A)$; see \cite[Lemma~3.3]{BIK}. We obtain an Ext-orthogonal
pair $(\X,\Y)$ for $\Mod A$ by taking $\X=H^0\Ker L$ and $\Y=H^0\Im
L$; see Proposition~\ref{pr:thick}. The fact that $L$ preserves
coproducts implies that $\Y$ is closed under taking coproducts. It
follows from Theorem~\ref{th:ext} that $\X$ is generated by finitely
presented modules.  Each finitely presented module is isomorphic in
$\bfD(\Mod A)$ to a perfect complex, and therefore $\Ker L$ is
generated by perfect complexes.

(2) $\Rightarrow$ (3): Suppose that $\C$ is generated by perfect
complexes. Then there exists a localization functor $L\colon\bfD(\Mod
A)\to\bfD(\Mod A)$ such that $\Ker L=\C$. Thus we have an
Ext-orthogonal pair $(\C,\D)$ for $\bfD(\Mod A)$ with $\D=\Im L$; see
\cite[Lemma~3.3]{BIK}.  Now observe that $\D=\C^\perp$ is closed under
coproducts, since for any perfect complex $X$ the functor
$\Hom_{\bfD(\Mod A)}(X,-)$ preserves coproducts. It follows that $\D$
is a localizing subcategory.

(3) $\Rightarrow$ (1): Let $\D$ be a localizing subcategory that is
closed under products such that $\C={^\perp}\D$. Then $\Y=H^0\D$ is an
extension closed abelian subcategory of $\Mod A$ that is closed under
products and coproducts; see Proposition~\ref{pr:thick-corr}. In the
proof of Corollary~\ref{co:perp} we have constructed a localization
functor $L\colon\bfD(\Mod A)\to\bfD(\Mod A)$ such that $\C=\Ker
L$. More precisely, there exists a homological epimorphism $A\to B$
such that $L=-\otimes_A^\bfL B$. It remains to notice that this
functor preserves coproducts.
\end{proof}

\begin{rem}
The implication (1) $\Rightarrow$ (2) is known as the \emph{telescope
conjecture}.  Let us sketch the essential ingredients of the proof of
this implication. In fact, the proof is not as involved as one might
expect from the references to preceding results of this work.

We need the 5-term exact sequence $\e_M$ for each module $M$ which one
gets immediately from the the localization functor $L$; see
Proposition~\ref{pr:local}. The perfect complexes generating $\C$ are
constructed in the proof of Theorem~\ref{th:ext}, where the relevant
implication is (1) $\Rightarrow$ (2). For this proof, one uses
Lemmas~\ref{le:ext1perp} -- \ref{le:subfp}, but this is all.
\end{rem}

\begin{rem}
Let $A$ be a hereditary ring and $B$ a ring that is derived equivalent
to $A$, that is, there is an equivalence of triangulated categories
$\bfD(\Mod A)\xto{\sim}\bfD(\Mod B)$. Then the statement of
Theorem~\ref{th:TC} carries over from $A$ to $B$. In particular, the
statement of Theorem~\ref{th:TC} holds for every tilted algebra in the
sense of Happel and Ringel \cite{HR}.
\end{rem}

Given the proof of the telescope conjecture for the derived categories
of hereditary rings, one may be tempted to think that perhaps it is
possible to get a similar result for rings of higher global
dimension. Here we show that this is not the case. Namely, we
construct a class of rings for which the conjecture fails for the
derived category, and we will see that some of them have global
dimension $2$. To achieve this, we use the following result due to
Keller \cite{Ke}.

\begin{lem}\label{le:tc_fail} 
Let $A$ be a ring and $I$ a non-zero two-sided ideal of $A$ such that
\begin{enumerate}
\item $\Tor_i^A(A/I,A/I) = 0$ for all $i \ge 1$ (that is, the
surjection $A \to A/I$ is a homological epimorphism), and
\item $I$ is contained in the Jacobson radical of $A$.
\end{enumerate}
Then $L = - \lotimes_A A/I\colon \bfD(\Mod A) \to \bfD(\Mod A)$ is a
coproduct preserving localization functor but $\Ker L$, which is the
smallest localizing subcategory containing $I$, contains no
non-zero perfect complexes. In particular, the telescope
conjecture fails for $\bfD(\Mod A)$.
\end{lem}

In order to find such $A$ and $I$ with (right) global dimension of $A$
equal to $2$, we restrict ourself to the case when $A$ is a valuation
domain. That is, $A$ is a commutative domain with the property that
for each pair $a, b \in A$, either $a$ divides $b$ or $b$ divides
$a$. We refer to~\cite[Chapter II]{FS} for a discussion of such
domains. Here, we mention only the properties which we need for our
example:

\begin{lem} \label{le:vd_basic}
The following holds for a valuation domain $A$ which is not a field.
\begin{enumerate}
\item The ring $A$ is local and its weak global dimension equals $1$.
\item The maximal ideal $P$ of $A$ is either principal or idempotent.
\item For any ideal $I$ of $A$ we have the isomorphism $\Tor_1^A(A/I,A/I) \cong I/I^2$.
\end{enumerate}
\end{lem}

\begin{proof}
(1) The ring $A$ is local since the ideals of $A$ are totally
    ordered by inclusion. The second part of (1) follows
    from~\cite[VI.10.4]{FS}.

(2) This is a direct consequence of results in~\cite[Section
    II.4]{FS}. For an ideal $I$, one defines
$$ I' = \{a \in A \mid aI \subsetneq I\}. $$
It turns out that $I'$ is always a prime ideal and $I$ is
naturally an $R_{I'}$-module. Moreover, $I = I'$ if $I$
itself is a prime ideal, \cite[II.4.3 (iv)]{FS}. In particular we have
$P' = P$. On the other hand, \cite[p.\ 69, item (d)]{FS} says
that $I' \cdot I \subsetneq I$ if and only if $I$ is a principal
ideal of $R_{I'}$. Specialized to $P$, this precisely says that
$P^2 = P' \cdot P \subsetneq P$ if and only if $P$ is a
principal ideal of $R$.

(3) Tensoring the exact sequence $0 \to I \to A \to A/I \to 0$ with
    $A/I$ gives the exact sequence
$$ A/I \otimes_A I \overset{0}\longrightarrow A/I
\overset{\sim}\longrightarrow A/I \otimes_A A/I \to 0. $$
It follows that $\Tor_1^A(A/I,A/I) \cong A/I \otimes_A I$, and the
right exactness of the tensor product yields $A/I \otimes_A I \cong
I/I^2$.
\end{proof}

The following result is a straightforward consequence.

\begin{prop} \label{pr:cntex}
Let $A$ be a valuation domain whose maximal ideal $P$ is
non-principal. Then the telescope conjecture fails for $\bfD(\Mod
A)$. More precisely, $L = - \lotimes_A A/P$ is a coproduct preserving
localization functor on $\bfD(\Mod A)$ whose kernel is non-trivial (it
contains $P$) but not generated by perfect complexes.
\end{prop}

\begin{proof}
It is enough to prove that the maximal ideal $P$ meets the conditions
of Lemma~\ref{le:tc_fail}. As $P$ is the Jacobson radical of $A$,
condition (2) is fulfilled. Condition (1) follows easily from
Lemma~\ref{le:vd_basic}.
\end{proof}

What we are left with now is to construct a valuation domain whose
maximal ideal is non-principal and whose global dimension is $2$. To
this end, we recall the basic tool to construct valuation domains with
given properties: the value group. If $A$ is a valuation domain,
denote by $Q$ its quotient field and by $U$ the group of units of
$A$. Then $U$ is clearly a subgroup of the multiplicative group $Q^* =
Q \setminus \{0\}$ and
$$ G = Q^* / U $$
is a totally ordered abelian group. More precisely, $G$ is an abelian
group, the relation $\leq$ on $G$ defined by $aU \leq bU$ if $ba^{-1}
\in A$ gives a total order on $G$, and we have the compatibility
condition
$$
\alpha \leq \beta \quad
\textrm{implies} \quad
\alpha \cdot \gamma \leq \beta \cdot \gamma \quad
\textrm{ for all } \alpha,\beta,\gamma \in G.
$$
The pair $(G,\leq)$ is called the \emph{value group} of
$A$. We will use the following fundamental result \cite[Theorem~3.8]{FS}.

\begin{prop}\label{pr:vd_constr} 
Let $k$ be a field and $(G, \leq)$ a totally ordered abelian
group. Then there is a valuation domain $A$ whose residue field $A/P$
is isomorphic to $k$, and whose value group is isomorphic to $G$ as an
ordered group.
\end{prop}

Now, we can give the promised example.

\begin{exm} \label{ex:tc_fail}
Let $G$ be a free abelian group of countable rank. If we view $G$ as
the group $\bbZ^{(\bbN)}$ (with additive notation), then $G$ is
naturally equipped with the lexicographic ordering which makes it to a
totally ordered group. Let $A$ be a valuation domain whose value group
is isomorphic to $G$. In fact, looking closer at the particular
construction in~\cite[Section II.3]{FS}, we can construct $A$ such
that it is countable.

We claim that the maximal ideal $P$ of $A$ is non-principal and
that $\gldim A = 2$. Indeed, each ideal of $A$ is flat and countably
generated since the value group is countable. Thus, each ideal is of
projective dimension at most $1$ and $\gldim A \le 2$. On the other
hand, it is easy to see that $A$ has non-principal, hence
non-projective, ideals and so is not hereditary. One of them is $P$,
which is generated by elements of $A$ whose cosets in the value group
$Q^*/U$ correspond, under the isomorphism $Q^*/U \cong \bbZ^{(\bbN)}$,
to the canonical basis elements $e_1, e_2, e_3, \ldots \in
\bbZ^{(\bbN)}$.

This way, we obtain a countable valuation domain $A$ of global
dimension $2$ such that the telescope conjecture fails for $\bfD(\Mod
A)$ by Proposition~\ref{pr:cntex}.
\end{exm}

\section{A bijective correspondence}\label{se:bij}

In this final section we summarize our findings by stating explicitly the
correspondence between various structures arising from Ext-orthogonal
pairs for hereditary rings. In particular, this completes the proof of an extended version of Theorem~B:

\begin{thm}\label{th:bijection}
For a hereditary ring $A$ there are bijections between the following
sets:
\begin{enumerate}
\item Ext-orthogonal pairs $(\X,\Y)$ for $\Mod A$ such that $\Y$ is
closed under coproducts.
\item Ext-orthogonal pairs $(\Y,\Z)$ for $\Mod A$ such that $\Y$ is
closed under products.
\item Extension closed abelian subcategories of $\Mod A$ that are
closed under products and coproducts.
\item Extension closed abelian subcategories  of $\mod A$.
\item Homological epimorphisms $A\to B$ (up to isomorphism).
\item Universal localizations $A\to B$ (up to isomorphism).
\item Localizing subcategories of $\bfD(\Mod A)$ that are closed under products.
\item Localization functors $\bfD(\Mod A)\to\bfD(\Mod A)$ preserving
coproducts (up to natural isomorphism).
\item Thick subcategories of $\bfD^b(\mod A)$.
\end{enumerate}
\end{thm}
\begin{proof}
We state the bijections explicitly in the following table and give the
references to the places where these bijections are established.

\begin{center}
\begin{tabular}{ccl}
Direction&Map&Reference\\ \hline \\[-1.8ex]
(1) $\leftrightarrow$ (3)& $(\X,\Y)\mapsto \Y$& Corollary~\ref{co:perp}\\
(2) $\leftrightarrow$ (3)& $(\Y,\Z)\mapsto \Y$& Corollary~\ref{co:perp}\\
(3) $\to$ (4)& $\Y\mapsto ({^\perp}\Y)\cap\mod A$& Thm.~\ref{th:ext} \& Prop.~\ref{pr:ext}\\
(4) $\to$ (3)& $\C\mapsto \C^\perp$& Thm.~\ref{th:ext} \& Prop.~\ref{pr:ext}\\
(3) $\to$ (5)& $\Y\mapsto (A\to \End_A(FA))$& Proposition~\ref{pr:hom-epi2}\\
(5) $\to$ (3)& $f\mapsto (\Ker f\oplus\Coker f)^\perp$& Proposition~\ref{pr:hom-epi1}\\
(5) $\leftrightarrow$ (6)& $f\mapsto f$& Theorem~\ref{th:epi}\\
(3) $\to$ (7)& $\Y\mapsto \bfD_\Y(\Mod A)$& Proposition~\ref{pr:thick-corr}\\
(7) $\to$ (3)& $\C\mapsto H^0\C$& Proposition~\ref{pr:thick-corr}\\
(7) $\to$ (8)& $\C\mapsto (X\mapsto GX)$& Theorem~\ref{th:TC}\\
(8) $\to$ (7)& $L\mapsto \Im L$& Theorem~\ref{th:TC}\\
(4) $\to$ (9)& $\X\mapsto \bfD^b_\X(\mod A)$& Remark~\ref{re:thick}\\
(9) $\to$ (4)& $\C\mapsto H^0\C$& Remark~\ref{re:thick} \\
\end{tabular}
\end{center}

\noindent
For (3) $\to$ (5), the functor $F$ denotes a left adjoint of the inclusion $\Y\to\Mod A$. For (7) $\to$ (8), the functor $G$ denotes a left adjoint of the inclusion $\C\to\bfD(\Mod A)$,
\end{proof}

Let us mention that this correspondence is related to recent work of
some other authors. In \cite{Sch2007}, Schofield establishes for any
hereditary ring the bijection (4) $\leftrightarrow$ (6).  In
\cite{NS}, Nicol\'as and Saor\'in establish for a differential graded
algebra $A$ a correspondence between recollements for the derived
category $\bfD(A)$ and differential graded homological epimorphisms
$A\to B$. This correspondence specializes for a hereditary ring to the
bijection (5) $\leftrightarrow$ (8).\footnote{The first author is
grateful to Manolo Saor\'in for pointing out this bijection.}

\subsection*{A finiteness condition}

Given an Ext-orthogonal pair for the category of $A$-modules as in
Theorem~\ref{th:bijection}, it is a natural question to ask when its
restriction to the category of finitely presented modules yields a
complete Ext-orthogonal pair for $\mod A$. This is very important
especially when considering relations of results from this paper
to representation theory of
finite dimensional algebras. For that setting, we characterize this
finiteness condition in terms of finitely presented modules;
see also Proposition~\ref{pr:except}.

\begin{prop}
Let $A$ be a finite dimensional hereditary algebra over a field and
$\C$ an extension closed abelian subcategory of $\mod A$. Then the
following are equivalent.
\begin{enumerate}
\item There exists a complete Ext-orthogonal pair $(\C,\D)$ for $\mod A$.
\item The inclusion $\C\to\mod A$ admits a right adjoint.
\item There exists an exceptional object $X\in\C$ such that $\C$ is
the smallest extension closed abelian subcategory of $\mod A$
containing $X$.
\item Let $(\X,\Y)$ be the Ext-orthogonal pair for $\Mod A$ generated
by $\C$. Then for each $M\in\mod A$ the 5-term exact sequence $\e_M$
belongs to $\mod A$.
\end{enumerate}
\end{prop}

\begin{proof}
(1) $\Rightarrow$ (2): For $M\in\mod A$ let $0\to D_M\to C_M\to M\to
D^M\to C^M\to 0$ be its 5-term exact sequence. Sending a module $M$ to $C_M$
induces a right adjoint for the inclusion $\C\to\mod A$; see
Lemma~\ref{le:exact}.

(2) $\Rightarrow$ (3): Choose an injective cogenerator $Q$ in $\mod A$
and let $X$ denote its image under the right adjoint of the inclusion
of $\C$.  A right adjoint of an exact functor preserves injectivity.
It follows that $X$ is an exceptional object and that $\C$ is the
smallest extension closed abelian subcategory of $\mod A$ containing
$X$.

(3) $\Rightarrow$ (4): See Proposition~\ref{pr:except}.

(4) $\Rightarrow$ (1): The property of the pair $(\X,\Y)$ implies that
$(\X\cap\mod A,\Y\cap\mod A)$ is a complete Ext-orthogonal pair for
$\mod A$. An application of Proposition~\ref{pr:ext} yields
the equality $\X\cap\mod A=\C$. Thus there exists a complete Ext-orthogonal pair
$(\C,\D)$ for $\mod A$.
\end{proof}

\begin{rem} 
There is a dual result which is obtained by applying the duality
between modules over the algebra $A$ and its opposite  $A^\op$. Note that condition (3)
is self-dual.
\end{rem}

\end{document}